\documentclass[10pt]{article}
\usepackage{amsmath,amssymb,amsfonts,amsthm,mathtools,epsfig}
\usepackage{tikz}
\usepackage{caption,subcaption,verbatim}
\usepackage{setspace}
\usepackage{algorithm}
\usepackage{algorithmic}
\usepackage{xurl}
\usepackage{hyperref}
\usepackage{ulem}
\usepackage{thmtools}

\usepackage{cancel}
\usepackage{comment}

\usepackage[bottom=1.5in]{geometry}

\usepackage[dvipsnames]{xcolor}
\newcommand{\cz}[1]{{\color{cyan} Carol: #1}}

\usetikzlibrary{arrows.meta,positioning,calc,decorations.pathreplacing}

\newtheorem{thm}{Theorem}
\newtheorem{theorem}[thm]{Theorem}
\newtheorem{lemma}[thm]{Lemma}
\newtheorem{observation}[thm]{Observation}
\newtheorem{corollary}[thm]{Corollary}

\newtheorem{proposition}[thm]{Proposition}

\theoremstyle{definition}

\newtheorem{problem}{Problem}
\newtheorem{question}[problem]{Question}
\newtheorem{conjecture}[problem]{Conjecture}

\date{ }

\title{\bf Counting Hamiltonian paths between prescribed vertices in traceable graphs with a forbidden induced subgraph}

\author{
Jorik Jooken \thanks{Department of Computer Science, KU Leuven Kulak, 8500 Kortrijk, Belgium.
}
\and Carol T. Zamfirescu
 \thanks{(1) Centre for Research in Mathematics and Data Science, Western Sydney University, Australia, and (2) Department of Mathematics, Computer Science and Statistics, Ghent University, 9000 Ghent, Belgium.\\ Email addresses:
 \protect\href{mailto:jorik.jooken@kuleuven.be}{\protect\nolinkurl{jorik.jooken@kuleuven.be}} and
\protect\href{mailto:czamfirescu@gmail.com}{\protect\nolinkurl{czamfirescu@gmail.com}}}
}

\date{}

\begin{document}
	
	\maketitle

	 \begin{abstract}
For graphs $G$ and $F$, we say that $G$ is \textit{$F$-free} if $F$ does not occur as an induced subgraph of $G$. This paper is concerned with the following question: Given an $F$-free graph $G$ having two vertices between which there exists at least one Hamiltonian path, how many Hamiltonian paths between these endpoints must exist (in terms of the order of $G$)? Our main result shows that there exists a sharp dichotomy. More precisely, we show that if $F$ is not an induced subgraph of $P_3+sP_1$ for any integer $s \geq 0$, then there exists an infinite family of $F$-free graphs having two vertices between which there exists a unique Hamiltonian path. On the other hand, we prove that if $F$ is an induced subgraph of $P_3+sP_1$ for some integer $s \geq 0$, then any $F$-free graph having two vertices between which there exists a Hamiltonian path contains exponentially many such paths between these two vertices. Our proofs use Ramsey-theoretic methods, a result on the existence of two vertices with low degree in graphs containing a unique Hamiltonian cycle, a path variant of Thomassen’s red-independent weakly green-dominating sets, and a structural analysis of Hamiltonian paths in $P_3+sP_1$-free graphs. As an algorithmic consequence we obtain that for every fixed $s \geq 1$, given a Hamiltonian $sP_1$-free graph together with a Hamiltonian cycle, one can decide in linear time whether a second Hamiltonian cycle exists and construct one if it does. 

			\vskip 3mm
		
			\end{abstract}

\newpage

    \section{Introduction}

    
    A graph is \textit{Hamiltonian} if it contains a spanning cycle, referred to as a \textit{Hamiltonian cycle}. Similarly, a graph is \textit{traceable} if it contains a spanning path, referred to as a \textit{Hamiltonian path}. Hamiltonicity and traceability are among the most studied notions in graph theory. One is often not only interested in the existence of a Hamiltonian cycle or path, but also in the extent to which additional assumptions force richer behavior. For example, such assumptions may guarantee cycles of many different lengths, a second Hamiltonian cycle or many Hamiltonian paths between prescribed vertices. A famous example is a problem that was stated by Erd{\H{o}}s~\cite{E72} in the seventies: Given a Hamiltonian graph $G$ on $n$ vertices for which the independence number $\alpha(G) \leq s$, how large does one have to make $n$ (in terms of $s$) to guarantee that $G$ contains a cycle of each length in $\{3,4,\ldots,n\}$? Erd{\H{o}}s~\cite{E72} showed that $\Omega(s^4)$ vertices suffice, but conjectured that $\Omega(s^2)$ vertices should already be enough. This lower bound was systematically improved over the last 50 years (see e.g.~\cite{D20,KS10,LS12}) until the conjecture was recently proven by Dragani{\'c}, Munh{\'a} Correia and Sudakov~\cite{DMS24} in a strong form by showing that $(2+o(1))s^2$ vertices suffice.

    In the current paper, we will be interested in the number of pairwise distinct Hamiltonian paths between prescribed vertices (and also obtain results on Hamiltonian cycles as direct corollaries). A famous example that deals with a similar notion is Sheehan's conjecture~\cite{S75} stating that every Hamiltonian $4$-regular graph contains at least two Hamiltonian cycles. 
    
    For integers $s \geq 2$ and $t \geq 1$, a graph is an \textit{$[s,t]$-graph} if every subgraph induced by $s$ vertices contains at least $t$ edges. This notion generalizes the independence number: for a graph $G$, we have $\alpha(G) \leq s$ if and only if $G$ is an $[s+1,1]$-graph. Several recent results investigate Hamiltonian properties under such conditions. Cheng~\cite{C26} gives additional background and further Hamiltonian-type results on $[s,t]$-graphs; in particular, he proved that for every $k$-connected $[k+1,2]$-graph there exists a Hamiltonian path between each two distinct vertices, apart from one explicit exceptional family. Li and Zhan~\cite{LZ26} recently showed that every Hamiltonian $[4,2]$-graph on at least eight vertices contains a second Hamiltonian cycle. The results that we will present in the current paper have as a corollary that for all integers $s \geq 2$ and $t \geq 1$, a Hamiltonian $[s,t]$-graph contains exponentially many Hamiltonian cycles.

    The current paper is situated in another 
    popular setting that generalizes the bounded independence number setting. For graphs $G$ and $F$, we say that $G$ is \textit{$F$-free} if $G$ does not contain $F$ as an induced subgraph. We write the disjoint union of $G$ and $F$ as $G+F$ and the disjoint union of $s$ copies of $F$ is written as $sF$. The path on $n$ vertices is denoted by $P_n$. Therefore, for a graph $G$ we have $\alpha(G) \leq s$ if and only if $G$ is $(s+1)P_1$-free. For vertices $a,b \in V(G)$, an \textit{$ab$-path} is a path with endpoints $a$ and $b$.
    
    Our main result is as follows.

    \begin{restatable}{theorem}{SummaryTheorem}
\label{th:summary}
Let $F$ be a graph. If $F$ is not an induced subgraph of $P_3+sP_1$ for any integer $s \geq 0$, then there exists an infinite family of $F$-free graphs, each with two vertices between which there exists a unique Hamiltonian path. On the other hand, if $F$ is an induced subgraph of $P_3+sP_1$ for some integer $s \geq 0$, then there exists a constant $c_s=2^{\Theta(s)}$ such that every $F$-free graph $G$ having a Hamiltonian path between two vertices $a, b \in V(G)$ contains at least $2^{\left\lfloor \frac{n}{c_s} \right\rfloor}$ distinct Hamiltonian $ab$-paths.
\end{restatable}

We also derive an algorithmic consequence for a search problem involving Hamiltonian cycles. More precisely, for every fixed integer $s \geq 1$, we show that the problem \textsc{Second Hamiltonian cycle in $sP_1$-free graphs} can be solved in linear time: given an $sP_1$-free graph $G$ together with a Hamiltonian cycle, the algorithm either constructs a second Hamiltonian cycle or correctly concludes that none exists.

    \section{Infinite families of $F$-free graphs having a unique Hamiltonian path between prescribed vertices}
    
    A graph $G$ is called \textit{uniquely Hamiltonian} if $G$ contains exactly one Hamiltonian cycle. We first observe the existence of two infinite families of uniquely Hamiltonian $F$-free graphs for particular graphs $F$. The next observation follows immediately by considering large enough cycles.
    \begin{observation}
    \label{obs:infFam1}
        Let $F$ be a graph that contains a cycle or a vertex with at least three neighbors. Then there exists an infinite family of $F$-free graphs with a unique Hamiltonian cycle.
    \end{observation}

    \begin{observation}
    \label{obs:infFam2}
        There exists an infinite family of $2P_2$-free graphs with a unique Hamiltonian cycle.
    \end{observation}
    \begin{proof}
        Let $k \geq 3$ be an integer. Let $G$ be a graph on $2k$ vertices obtained by taking a clique $K_k$ on $k$ vertices, duplicating each edge from a Hamiltonian cycle $\mathfrak{C}$ in $K_k$ and subdividing once each edge in $\mathfrak{C}$ ($G$ is also known as a \textit{sun graph}). Since every Hamiltonian cycle must contain the two edges incident with a vertex of degree two, $G$ contains a unique Hamiltonian cycle. Moreover, $G$ is $2P_2$-free, since every edge in $G$ contains at least one vertex from $K_k$. \end{proof}







Since the above infinite families of graphs are uniquely Hamiltonian, they contain a linear number of pairs of vertices between which there exists a unique Hamiltonian path. For the next infinite family of graphs, we can only guarantee a constant number of such pairs of vertices.

\begin{observation}
    \label{obs:infFam3}
        There exists an infinite family of $P_4$-free graphs such that each graph in this family has two unordered pairs of vertices between which there exists a
unique Hamiltonian path.
    \end{observation}
\begin{proof}
Let $k \geq 2$ and let $G$ be a graph on $2k$ vertices such that $V(G)$ can be partitioned into sets $X$ and $Y$, where $X := \{x_1,x_2,\ldots,x_k\}$ induces an independent set and $Y := \{y_1,y_2,\ldots,y_k\}$ induces a clique (i.e., $G$ is a split graph). The edge set of $G$ is given by $E(G):=\{y_iy_j~|~1\leq i <j \leq k\} \cup \{x_iy_j~|~1\leq j \leq i \leq k\}$. An example of the graph obtained by taking $k=4$ is given in Fig.~\ref{fig:P4FreeK4}.

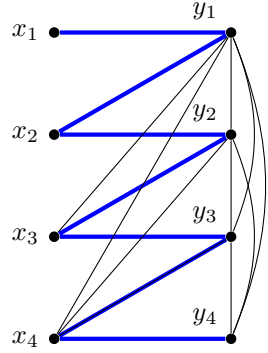
\begin{figure}[ht]
\centering
\begin{tikzpicture}[
    x=1.3cm,y=1cm,
    block/.style={draw,rounded corners,minimum width=1.0cm,minimum height=.55cm},
    dot/.style={circle,fill,inner sep=1.4pt},
    >=Latex,scale=0.9
]

    \node[dot,label=left:{$x_1$}] (x1) at (0,0) {};
    \node[dot,label=left:{$x_2$}] (x2) at (0,-1.5) {};
    \node[dot,label=left:{$x_3$}] (x3) at (0,-3) {};
    \node[dot,label=left:{$x_4$}] (x4) at (0,-4.5) {};

    \node[dot,label=above left:{$y_1$}] (y1) at (2,0) {};
    \node[dot,label=above left:{$y_2$}] (y2) at (2,-1.5) {};
    \node[dot,label=above left:{$y_3$}] (y3) at (2,-3) {};
    \node[dot,label=above left:{$y_4$}] (y4) at (2,-4.5) {};

    \draw[ultra thick,blue] (x1) -- (y1) -- (x2) -- (y2) -- (x3) -- (y3) -- (x4) -- (y4);

    \draw (y1)--(y2)--(y3)--(y4);
    \draw (y1) to[bend left=22] (y3);
    \draw (y2) to[bend left=22] (y4);
    \draw (y1) to[bend left=22] (y4);

    \draw (x3)--(y1);

    \draw (x4)--(y1);
    \draw (x4)--(y2);
    \draw (x4)--(y3);
    
\end{tikzpicture}
\caption{The graph $G$ obtained by taking $k=4$. The unique Hamiltonian $x_1y_4$-path is marked in blue.}
\label{fig:P4FreeK4}
\end{figure}

The graph $G$ is clearly $P_4$-free, since any path on four vertices in $G$ contains two distinct vertices $y_i,y_j \in Y$ and at least one of the other two vertices must be adjacent to both $y_i$ and $y_j$. Therefore, these four vertices do not induce a $P_4$. Moreover, we claim that there exists a unique Hamiltonian $x_1y_k$-path, namely $x_1y_1x_2y_2 \ldots x_{k-1}y_{k-1}x_ky_k$ 
We prove by induction on~$i$ that for every $1 \leq i \leq k$, every Hamiltonian $x_1y_k$-path contains $x_1y_1x_2y_2 \ldots x_{i}y_{i}$ as a subpath. The base case $i=1$ is trivial, since $x_1$ has degree $1$. Now, assume $2 \leq i \leq k$ and assume the statement holds for $i-1$. Since $N(x_i) \setminus \{x_1,y_1,\ldots,x_{i-2},y_{i-2},x_{i-1}\}=\{y_{i-1},y_i\}$, we see that $x_1y_1x_2y_2 \ldots y_{i-1}x_{i}y_{i}$ must be a subpath of any Hamiltonian $x_1y_k$-path. By taking $i=k$, we conclude that there is a unique Hamiltonian $x_1y_k$-path. A nearly identical argument shows that there is a unique Hamiltonian $x_1x_k$-path.
\end{proof}


    These observations imply that the only graphs $F$ for which we do not know yet whether there exist infinitely many $F$-free graphs having two vertices between which there exists a unique Hamiltonian path are those which are induced subgraphs of $P_3+sP_1$ for some non-negative integer~$s$. This is the focus of the next section. 
    
    \section{Exponentially many Hamiltonian paths between prescribed vertices in traceable $F$-free graphs}\label{sec:expCrudeBound}

    A class of graphs $\mathcal{G}$ is called \textit{hereditary} if for every graph $G \in \mathcal{G}$, every induced subgraph of $G$ is also in $\mathcal{G}$. In particular, it is worth noting that for every graph $F$, the class of $F$-free graphs is hereditary. We first introduce a lemma that relates Hamiltonian paths in small graphs in a hereditary class of graphs with Hamiltonian paths in large such graphs.
    
    \begin{lemma}
    \label{lem:pathToCycles}
        Let $\mathcal{G}$ be a hereditary class of graphs and let $G \in \mathcal{G}$ be a graph on $n$ vertices having a Hamiltonian path between $v_1, v_n \in V(G)$. If there exist integers $n'$ and $c$ such that every graph $G' \in \mathcal{G}$ on $n'$ vertices having a Hamiltonian path between two distinct vertices $a, b \in V(G')$ has in fact at least $c$ Hamiltonian $ab$-paths, then $G$ has at least $c^{\left\lfloor n/n'\right\rfloor}$ distinct Hamiltonian $v_1v_n$-paths. 
    \end{lemma}
    \begin{proof}
    If $n<n'$, the statement is trivial. Otherwise, let $\mathfrak{P} :=  v_1v_2 \ldots v_n$ be a Hamiltonian $v_1v_n$-path in $G$ and let $G' := G[\{v_1,v_2,\ldots,v_{n'}\}]$. Since $G \in \mathcal{G}$ and $\mathcal{G}$ is hereditary, we also have $G' \in \mathcal{G}$. Therefore, $G'$ has at least $c$ Hamiltonian $v_1v_{n'}$-paths (and the same conclusion holds for any graph induced by $n'$ consecutive vertices in $\mathfrak{P}$). Now, one obtains $c^{\left\lfloor \frac{n}{n'}\right\rfloor}$ Hamiltonian $v_1v_n$-paths in $G$ by repeatedly replacing $n'$ consecutive vertices in $\mathfrak{P}$ by one of the $c$ Hamiltonian paths between the endpoints of the $n'$ consecutive vertices.
    \end{proof}

    We now recall a theorem due to Abbasi and Jamshed~\cite{AJ06}. Throughout this paper, let $\lambda := \frac{1}{\log_2(3)-1} \approx 1.7095$.

\begin{theorem}[Abbasi and Jamshed~\cite{AJ06}, Corollary 3]
\label{th:lowDegreeVertexForUniqueHamCycle}
  Let $G$ be a uniquely Hamiltonian graph on $n$ vertices. Then $G$ contains two vertices of degree at most $\lambda \log_2(n)+4$.
\end{theorem}

It is not hard to show the following consequence for graphs with a unique Hamiltonian path between two vertices.

\begin{lemma}
\label{lem:lowDegreeVertexForUniqueHamPath}
    Let $G$ be a graph on $n$ vertices and let $a,b \in V(G)$ be two distinct vertices such that $G$ has a unique Hamiltonian $ab$-path. Then $G$ contains a vertex of degree at most $\lambda \log_2(n+1)+4$.
\end{lemma}
\begin{proof}
    Let $G'$ be the graph obtained by adding a vertex $c$ to $G$ and adding the edges $ac$ and $bc$. Since $c$ has degree 2, every Hamiltonian cycle in $G'$ must contain the edges $ac$ and $bc$. Therefore, $G'$ has a unique Hamiltonian cycle. By Theorem~\ref{th:lowDegreeVertexForUniqueHamCycle}, $G'$ contains two vertices of degree at most  $\lambda \log_2(n+1)+4$. Therefore, $G$ contains a vertex of degree at most $\lambda \log_2(n+1)+4$.
\end{proof}

We now prove that large $P_3+sP_1$-free graphs containing a Hamiltonian path between two vertices must contain at least two such paths.

\begin{theorem}
Let $s \geq 0$ be an integer and let $G$ be a $P_3+sP_1$-free graph on $n$ vertices such that there is a unique Hamiltonian path between $v_1, v_n \in V(G)$, namely $\mathfrak{P} := v_1 v_2 \ldots v_n$. Then there exists a constant $M_s$ (only depending on $s$) such that $n \leq M_s$. In particular, one may take $M_0 := 3$ and $M_s := \left\lfloor 7\lambda(M_{s-1}+1)\log_2(7\lambda(M_{s-1}+1))\right\rfloor$ for $s \ge 1$.
\end{theorem}
\begin{proof}

We first show that we may take $M_0=3$. Suppose $n \geq 4$. Since $G$ is $P_3$-free, the triples
$\{v_1,v_2,v_3\}$ and $\{v_2,v_3,v_4\}$ force the edges $v_1v_3$ and
$v_2v_4$. Therefore, replacing the subpath
$v_1 v_2 v_3 v_4$ by $v_1 v_3 v_2 v_4$ yields a second Hamiltonian $v_1v_n$-path, a contradiction.

Hence, we have shown that the statement holds for $s=0$ and we now prove the remaining cases by induction on $s$. Let $s \geq 1$ and assume the statement holds for $s-1$. By Lemma~\ref{lem:lowDegreeVertexForUniqueHamPath}, $G$ contains a vertex $x$ of degree at most $\lambda \log_2(n+1)+4$. Let $G' := G-N[x]$ be the graph obtained by removing $x$ and all its neighbors from $G$. Note that $G'$ is $P_3+(s-1)P_1$-free, because if $G'$ would contain a vertex set $S' \subseteq V(G')$ that induces a $P_3+(s-1)P_1$, then $S' \cup \{x\}$ would induce a $P_3+sP_1$ in $G$. Let $A_1,\ldots,A_q$ be the subsets of $V(G')$ inducing the connected components of $\mathfrak{P}-N[x]$.

Note that $q \leq \deg_G(x)+2$. Now, for each $i \in [q]$, $G[A_i]$ is $P_3+(s-1)P_1$-free, because $G'$ is $P_3+(s-1)P_1$-free. Moreover, $G[A_i]$ contains a unique Hamiltonian path between the endpoints of the path $\mathfrak{P}[A_i]$, because otherwise $G$ would contain at least two Hamiltonian $v_1v_n$-paths by replacing the corresponding subpath in $\mathfrak{P}$. Therefore, we can apply the induction hypothesis and $|V(G[A_i])| \leq M_{s-1}$ for all $i \in [q]$. Combining this with the aforementioned bound on $q$ we have $$n \leq 1+\deg_G(x)+qM_{s-1} \leq 1+\deg_G(x)+(\deg_G(x)+2)M_{s-1} \leq (\deg_G(x)+2)(M_{s-1}+1).$$ Therefore
$$n \leq (\lambda \log_2(n+1)+6)(M_{s-1}+1).$$

This inequality has a function depending on $n$ on both sides, but we would like to bound $n$ only in terms of $M_{s-1}$. So the natural question is then: how large can we make $n$ before this inequality starts failing? This question motivates the next part of the proof.

Define $r := 7\lambda (M_{s-1}+1)$. We claim that $n<r \log_2(r)$. Consider the function $$h(t) := t-(\lambda \log_2(t+1)+6)(M_{s-1}+1).$$

By definition of $h(t)$, we have $h(n) \leq 0$. Our goal is to show that $h(t)>0$ for every $t \geq r \log_2(r)$ (and therefore, we must have $n<r \log_2(r)$). The derivative of $h(t)$ is given by:
$$h'(t)=1-\frac{\lambda(M_{s-1}+1)}{(t+1) \ln(2)}.$$

Since $M_{s-1}+1 \geq M_0+1=4$, we have $r = 7\lambda (M_{s-1}+1)>2$ and hence $r \log_2(r) > r$. In turn, we obtain for every $t \geq r \log_2(r)$ that $h'(t) \geq 1-\frac{\lambda(M_{s-1}+1)}{r \ln(2)}=1-\frac{1}{7 \ln(2)}>0$. Therefore, $h(t)$ is increasing on the interval $[r \log_2(r),\infty)$.

Finally, we will show that $h(r\log_2(r)) >0$. Since $r \log_2(r) \geq 1$, we have $r \log_2(r)+1 \leq 2r\log_2(r)$ and thus $\log_2(r \log_2(r)+1) \leq 1+\log_2(r)+\log_2(\log_2(r))$. Therefore, we have
$$
\begin{aligned}
h(r \log_2(r)) &= r\log_2(r)-\bigl(\lambda \log_2(r \log_2(r)+1)+6\bigr)(M_{s-1}+1) \\
&\geq r\log_2(r)-\bigl(\lambda(1+\log_2(r)+\log_2(\log_2(r)))+6\bigr)(M_{s-1}+1)\\
&= (7\lambda(M_{s-1}+1))\log_2(r)-\bigl(\lambda(1+\log_2(r)+\log_2(\log_2(r)))+6\bigr)(M_{s-1}+1)\\
&= (M_{s-1}+1)(6\lambda\log_2(r)-\lambda \log_2(\log_2(r))-\lambda-6).
\end{aligned}
$$
We have $r>2$, so $\log_2(\log_2(r)) \leq \log_2(r)$ and therefore
$$
h(r \log_2(r)) \geq (M_{s-1}+1)(5\lambda\log_2(r)-\lambda-6).
$$
Moreover, we have $\log_2(r)\geq1$ and thus 
$$
h(r \log_2(r)) \geq (M_{s-1}+1)(4\lambda-6).
$$
Since $4\lambda-6>0$, we conclude that indeed
$$
h(r \log_2(r)) >0.
$$
Therefore
$$
n<r\log_2(r)=7\lambda (M_{s-1}+1)\log_2(7\lambda (M_{s-1}+1)) 
$$
and thus
$$
n \leq \left\lfloor 7\lambda(M_{s-1}+1)\log_2(7\lambda(M_{s-1}+1))\right\rfloor=M_s.
$$
\end{proof}

By combining the previous theorem with Lemma~\ref{lem:pathToCycles}, we get the following corollary.

\begin{corollary}
\label{cor:countingsP4PlussP1}
Let $s \geq 0$ be an integer and let $G$ be a $P_3+sP_1$-free graph on $n$ vertices having two vertices $a,b \in V(G)$ between which there exists a Hamiltonian path. Then $G$ contains at least $2^{\left\lfloor \frac{n}{M_s+1}\right\rfloor}$ distinct Hamiltonian $ab$-paths.
\end{corollary}

Finally, we draw the reader's attention to the asymptotic behavior of $M_s$.

\begin{proposition}
\label{prop:growthOfMs}
    Let $c>1$ be a constant and define a sequence $(M_s)_{s\ge 0}$ by
$$
M_0=3,
\qquad
M_s=\left\lfloor c(M_{s-1}+1)\log_2\!\bigl(c(M_{s-1}+1)\bigr)\right\rfloor
\quad\text{for }s\ge 1.
$$
Define $a_s:=\log_2(M_s+1)$. Then, $a_s=\Theta(s\log s)$ and thus $M_s=2^{\Theta(s\log s)}$.

\end{proposition}

\begin{proof}

Since $M_0=3$ and $(M_s)$ is increasing, we have $M_s+1 \geq 4$ for all $s \geq 0$. Hence, $X_s := c(M_{s-1}+1) \log_2(c(M_{s-1}+1)) \geq 4c \log_2(4c)>2$ for all $s \geq 1$. By definition we have $M_s = \lfloor X_s \rfloor$ and therefore $X_s \leq M_s+1 \leq X_s+2 \leq 2X_s$ and $\log_2(X_s) \leq a_s \leq \log_2(X_s)+1$. Since $\log_2(X_s)=\log_2(c)+\log_2(M_{s-1}+1)+\log_2(\log_2(c(M_{s-1}+1)))$, we have $a_s=a_{s-1}+\log_2(a_{s-1})+O(1)$. So there exists a constant $K>0$ such that for all $s\ge 1$,
\begin{equation}
\label{eq:rec-as}
a_{s-1}+\log_2(a_{s-1})
\leq 
a_s
\leq
a_{s-1}+\log_2(a_{s-1})+K.
\end{equation}

We first show that $a_s=\Omega(s\log(s))$.
Since $a_0=\log_2(4)>1$, we have $\log_2(a_{s-1}) \geq \log_2(\log_2(4))>0$ for all $s\ge 1$.
Hence, \eqref{eq:rec-as} gives $a_s \geq a_{s-1}+\log_2(\log_2(4))$, and so $a_s=\Omega(s)$. Therefore, there exists a constant $K_1>0$ such that
$a_{s-1}\geq \frac{s}{K_1}$ for all $s\geq 1$. Hence, we have $\log_2(a_{s-1}) \geq \log_2(s)-\log_2(K_1)$. Using the first inequality in \eqref{eq:rec-as}, we obtain $a_s\geq a_{s-1}+\log_2(s)-\log_2(K_1)$. Summing this inequality for $1,2,\dots,s$ yields $a_s\geq a_0+\sum_{i=1}^{s}(\log_2(i))- \log_2(K_1)s$. By Stirling's formula, $\sum_{i=1}^{s} \log_2(i) = \log_2(s!) = s\log_2(s)+O(s)$ and therefore $a_s=\Omega(s\log(s))$.

We now prove the upper bound $a_s=O(s\log s)$.
Choose a constant $K_2>0$ large enough such that
\begin{equation}
\label{eq:A-big}
K_2\log_2(s+1) \geq \log_2(K_2(s+1)\log_2(s+2))+K
\qquad\text{for all }s \geq 1.
\end{equation}
We claim that $a_s\leq K_2(s+1)\log_2(s+2)$ for all $s \geq 0$. We show this by induction on $s$. If $K_2$ is large enough, this inequality clearly holds for $s=0$. Assume now that $s \geq 1$ and that $a_{s-1}\le K_2 s\log_2(s+1)$. By \eqref{eq:rec-as}, we have $a_s\leq K_2 s\log_2(s+1)+\log_2(K_2 s\log_2(s+1))+K$.

By \eqref{eq:A-big}, we have $a_s\leq K_2 s\log_2(s+1)+K_2\log_2(s+1)$. Since $K_2(s+1)\log_2(s+2)-K_2 s\log_2(s+1) \geq K_2\log_2(s+1)$, we have $a_s \leq K_2(s+1) \log_2(s+2)$ and therefore $a_s=O(s \log(s))$.
\end{proof}

The arguments in the current section yield a comparatively direct inductive proof of exponential growth in the number of Hamiltonian paths. In the next section, we develop a substantially more technical approach, partially based on a powerful technique due to Thomassen. It leads to an improved exponent and, as a byproduct, yields additional structural properties of Hamiltonian paths with algorithmic consequences (see Section~\ref{sec:hamCycles}).

\section{Improving the exponent}

    We first define the concept of a $2$-switch. Let $\mathfrak{P}=v_1v_2\cdots v_n$ be a Hamiltonian path of a graph $G$. We say that a Hamiltonian path $\mathfrak{P}'$ is obtained from $\mathfrak{P}$ by a \textit{$2$-switch} if there exist integers $i$ and $j$ such that $1\leq i<j-1\leq n-2$ and $v_iv_j,\; v_{i+1}v_{j+1}\in E(G)$ and $\mathfrak{P}'$ is obtained from $\mathfrak{P}$ by replacing the edges $v_iv_{i+1}$ and $v_jv_{j+1}$ with the edges $v_iv_j$ and $v_{i+1}v_{j+1}$. Note that $\mathfrak{P}'$ is also a Hamiltonian $v_1v_n$-path.
    
    We now show that large $sP_1$-free graphs with a Hamiltonian path between two vertices have another Hamiltonian path between these vertices obtained by a 2-switch. In the following theorem, the notation $R(s,t)$ denotes the Ramsey number indicating the smallest integer $n$ such that every graph on at least $n$ vertices contains an independent set with $s$ vertices or a clique with $t$ vertices.

    \begin{theorem}
    \label{th:sP1}
        Let $s \geq 1$ be an integer and let $G$ be an $sP_1$-free graph on $R(s,s+3)$ vertices having a Hamiltonian path $\mathfrak{P}$ between two vertices $a, b \in V(G)$. Then $G$ has a second Hamiltonian $ab$-path $\mathfrak{P}'$, obtained from $\mathfrak{P}$ by a $2$-switch.
    \end{theorem}
    \begin{proof}
    Let $n := R(s,s+3)$. Since $G$ has $n$ vertices and $G$ is $sP_1$-free, $G$ must contain a clique $K$ on $s+3$ vertices as a subgraph. Let 
    $v_1 := a$, $v_n := b$ and let $\mathfrak{P} := v_1 v_2 \ldots v_n$ be a Hamiltonian $ab$-path. Suppose for the sake of obtaining a contradiction that $G$ does not have a second Hamiltonian $ab$-path $\mathfrak{P}'$, obtained from $\mathfrak{P}$ by a 2-switch. We will now consider how $\mathfrak{P}$ traverses $K$. 
    
    First note that there do not exist integers $i$ and $j$ such that $1 \leq i \leq j-2 \leq n-3$ and $v_iv_{i+1} \in E(K)$ and $v_jv_{j+1} \in E(K)$, because then $\mathfrak{P}' := v_1 v_2 \ldots v_i v_jv_{j-1} \ldots v_{i+1} v_{j+1} v_{j+2} \ldots v_n$ is a second Hamiltonian $ab$-path, obtained from $\mathfrak{P}$ by a 2-switch. In particular, this implies that at most three consecutive vertices in $\mathfrak{P}$ all belong to $K$. Because of the two preceding sentences and since $K$ contains $s+3$ vertices, there must be a subpath of $\mathfrak{P}$ of the form $B_1A_1B_2 \ldots A_{s}B_{s+1}$, where each $A_i$ $(i \in [s])$ represents a non-empty subpath of $\mathfrak{P}$ containing exclusively vertices from $V(G) \setminus V(K)$ and each $B_i$ $(i \in [s+1])$ represents a non-empty subpath of $\mathfrak{P}$ containing exclusively vertices from $V(K)$. 
    
    For each $i \in [s]$ we let $A_i[-1]$ denote the last vertex on the subpath $A_i$; we let $A_i^{-}$ denote the (potentially empty) subpath of $A_i$ that contains all vertices of $A_i$ except for $A_i[-1]$; and for a path $P := p_1 p_2 \ldots p_k$, we let $\overleftarrow{P}$ denote the reversed path $p_k p_{k-1} \ldots p_1$. 

We now claim that $\{A_1[-1],A_2[-1],\ldots,A_s[-1]\}$ is an independent set, which contradicts the fact that $G$ is $sP_1$-free. Suppose for the sake of obtaining a contradiction that there are two integers $1 \leq i <j \leq s$ such that there is an edge between $A_i[-1]$ and $A_j[-1]$. Then the path obtained by replacing the subpath 
$$B_1A_1B_2 \ldots A_sB_{s+1}$$ with 
$$B_1A_1 \ldots A_i^{-} A_i[-1]A_{j}[-1]\overleftarrow{A_j^{-}}\overleftarrow{B_{j}}\overleftarrow{A_{j-1}} \ldots \overleftarrow{B_{i+1}} B_{j+1} A_{j+1} \ldots B_{s+1}$$ yields a second Hamiltonian $ab$-path $\mathfrak{P}'$ in $G$, obtained from $\mathfrak{P}$ by a 2-switch. Contradiction.

    \end{proof}

By combining Lemma~\ref{lem:pathToCycles} and Theorem~\ref{th:sP1}, we obtain the following corollaries.

\begin{corollary}
\label{cor:countingsP1}
    Let $s \geq 1$ be an integer and let $G$ be an $sP_1$-free graph on $n$ vertices having two vertices $a,b \in V(G)$ between which there exists a Hamiltonian path. Then $G$ contains at least $2^{\left\lfloor \frac{n}{R(s,s+3)}\right\rfloor}$ distinct Hamiltonian $ab$-paths.
\end{corollary}

\begin{corollary}
\label{cor:stGraphs}
    Let $s \geq 2$ and $t \geq 1$ be integers and let $G$ be an $[s,t]$-graph on $n$ vertices having two vertices $a,b \in V(G)$ between which there exists a Hamiltonian path. Then $G$ contains at least $2^{\left\lfloor \frac{n}{R(s,s+3)}\right\rfloor}$ distinct Hamiltonian $ab$-paths.
\end{corollary}

We now introduce a lemma due to Thomassen (see~\cite{BJ98,T97}). Let $G$ be a graph for which all edges are colored green and red. We say that $u, v \in V(G)$ are \textit{red neighbors} if $uv$ is a red edge and \textit{green neighbors} if $uv$ is a green edge. We call a set $X \subseteq V(G)$ \textit{red-independent} if no two vertices $u, v \in X$ are red neighbors and \textit{weakly green-dominating} if every red neighbor of each vertex in $X$ also has a green neighbor in $X$.

\begin{lemma}[Thomassen~\cite{T97}]
\label{lem:CycleRedGreen}
    Let $G$ be a graph with a Hamiltonian cycle $\mathfrak{C}$. Color the edges of $E(\mathfrak{C})$ red and color the edges of $E(G)\setminus E(\mathfrak{C})$ green. If there exists a non-empty set $X \subseteq V(G)$ that is both red-independent and weakly green-dominating, then $G$ contains a second Hamiltonian cycle $\mathfrak{C}'$ distinct from $\mathfrak{C}$.
\end{lemma}

We now prove a \lq Hamiltonian path version\rq~of this lemma.

\begin{lemma}
\label{lem:PathRedGreen}
    Let $G$ be a graph with a Hamiltonian path $\mathfrak{P}$ between two vertices $a$ and $b$. Color the edges of $E(\mathfrak{P})$ red and color the edges of $E(G)\setminus E(\mathfrak{P})$ green. If there exists a non-empty set $X \subseteq V(G) \setminus \{a,b\}$ that is both red-independent and weakly green-dominating, then $G$ contains a second Hamiltonian $ab$-path $\mathfrak{P}'$ distinct from $\mathfrak{P}$.
\end{lemma}
\begin{proof}
Let $G'$ be the graph obtained by adding a vertex $c$ to $G$ and adding the edges $ac$ and $bc$. Then the edge set $E(\mathfrak{P}) \cup \{ac,bc\}$ induces a Hamiltonian cycle $\mathfrak{C}$ in $G'$. Since $X \cap \{a,b,c\}=\emptyset$, the set $X$ is still red-independent and weakly green-dominating in $G'$. Therefore, by Lemma~\ref{lem:CycleRedGreen} $G'$ has a second Hamiltonian cycle $\mathfrak{C}'$ distinct from $\mathfrak{C}$. Since $c$ has degree 2, every Hamiltonian cycle in $G'$ must contain the edges $ac$ and $bc$. Therefore, the path $\mathfrak{P}'$ obtained by removing $c$ from $\mathfrak{C}'$ yields a second Hamiltonian $ab$-path $\mathfrak{P}'$ in $G$. \end{proof}

We are now ready to prove the following theorem that will almost directly lead to a generalization of Theorem~\ref{th:sP1}.

\begin{theorem}
\label{th:mainTheorem}
    Let $s \geq 0$ be an integer and let $G$ be a $P_3+sP_1$-free graph on $n$ vertices such that there is a unique Hamiltonian path between $v_1, v_n \in V(G)$, namely $\mathfrak{P} = v_1 v_2 \ldots v_n$. Then $\alpha(G) \leq 5s+6$.
\end{theorem}
\begin{proof}
If $s=0$, then $G$ is $P_3$-free. Since $G$ has a Hamiltonian path, $G$ is connected, and every connected $P_3$-free graph is complete. Hence $\alpha(G)=1$, and the claim follows. We may therefore assume that $s\ge 1$.

Color the edges of $E(\mathfrak{P})$ red and the edges of $E(G) \setminus E(\mathfrak{P})$ green. Let $k := \alpha(G)$ and let $I := \{u_1,u_2,\ldots,u_k\}$ be a maximum independent set in $G$ (such that the vertices $u_i$ are ordered in the order that they appear on $\mathfrak{P}$) that maximizes the sum of the indices  of these $k$ vertices. If $k \leq s+1$, there is nothing to prove, so assume $k \geq s+2$. We call a vertex $x \in V(G) \setminus V(I)$ \textit{bad} if $x$ has no green neighbor in $I$. We first prove the following claim.

\bigskip 

\noindent \textbf{Claim 1.} Let $x \in V(G) \setminus V(I)$ be a vertex that is not bad and also a red neighbor of some vertex $u_i \in I$. Then $x$ has at least $k-s-1$ green neighbors in $I$.

\bigskip

\noindent Since $x$ is not bad, it has a green neighbor $u_t \in I$. This means that $\{u_i,x,u_t\}$ induces a $P_3$, because $I$ is an independent set. Let $N \subseteq I$ be the set of vertices from $I$ that are not adjacent to $x$. If $|N| \geq s$, then $G[N \cup \{u_i,x,u_t\}]$ would contain an induced $P_3+sP_1$ and therefore $G$ would not be $P_3+sP_1$-free. As a result, we have $|N| \leq s-1$ and therefore $x$ is adjacent to at least $k-s+1$ vertices in $I$. Since $x$ is adjacent to at most two red vertices, $x$ has at least $k-s-1$ green neighbors in $I$. This proves the claim. 

\bigskip

The following claim provides more insight into bad vertices.

\bigskip 

\noindent \textbf{Claim 2.} Every bad vertex $y \in V(G) \setminus V(I)$ is adjacent to exactly one vertex in $I$.

\bigskip

\noindent

Since $I$ is a maximum independent set, every vertex in $V(G) \setminus V(I)$ is adjacent to at least one vertex in $I$. Suppose for the sake of obtaining a contradiction that $y$ would be adjacent to two vertices $u_i, u_j \in I$. Since $I$ is an independent set, $\{u_i,y,u_j\}$ induces a $P_3$. Since $y$ is bad, $u_i$ and $u_j$ are both red neighbors of $y$. But then $G[\{u_i,y,u_j\} \cup (I \setminus \{u_i,u_j\})]$ contains an induced $P_3+sP_1$ since $|I \setminus \{u_i,u_j\}| = k-2 \geq s$. Contradiction. This proves the claim. 

\bigskip

The following claim justifies why $I$ was chosen such that the sum of the positions on $\mathfrak{P}$ is maximized.

\bigskip 

\noindent \textbf{Claim 3.} Let $u_i \in I$ and suppose there exists a successor $x$ of $u_i$ on the path $\mathfrak{P}$. Then $x$ is not bad. 

\bigskip

\noindent

Suppose for the sake of obtaining a contradiction that $x$ is bad. By Claim 2, the only neighbor of $x$ in $I$ is $u_i$. But then $(I \cup \{x\}) \setminus \{u_i\}$ would be a maximum independent set for which the sum of the positions of the vertices is strictly larger than the sum of the positions of the vertices in $I$. This contradicts the choice of $I$ and proves the claim.

\bigskip

If $x \in V(G) \setminus V(I)$ is bad then there must be a successor $u$ of $x$ in the path $\mathfrak{P}$ by Claims~2 and 3. For each $i \in \{1,2,\ldots,k\}$, let $x_i$ be the predecessor of $u_i$ on the path $\mathfrak{P}$ if it exists. Define $B := \{i~|~i \in \{1,2,\ldots,k\}\text{ and }x_i\text{ exists and is bad}\}$. Define $X := \{u_i~|~ i \in (\{1,2,\ldots,k\} \setminus B)\text{ and }u_i \notin \{v_1,v_n\}\}$. The next claim states that $k$ must be bounded by a function depending only on $|B|$ and $s$.

\bigskip 

\noindent \textbf{Claim 4.} We have $k \leq |B|+s+3$.

\bigskip

\noindent

Suppose for the sake of obtaining a contradiction that $k>|B|+s+3$. Then $X$ is non-empty. Let $u_i \in X$ be a vertex in $X$. Since $u_i \notin \{v_1,v_n\}$, it has two red neighbors on $\mathfrak{P}$. Claim 3 and the choice of $X$ imply that neither of these two red neighbors are bad. By Claim 1, each of these two red neighbors has at least $k-s-1$ green neighbors in $I$. Since $k-s-1>|B|+2$ and $|X| \geq k-|B|-2$, each of these two red neighbors has a green neighbor in $X$. But then $X$ is a red-independent and weakly green-dominating set that contains neither $v_1$ nor $v_n$ and by Lemma~\ref{lem:PathRedGreen}, $G$ would not contain a unique Hamiltonian $v_1v_n$-path. This leads to a contradiction and proves the claim.

\bigskip

Next, we will work towards bounding $|B|$ in terms of $s$. Let $i \in B$. By Claim~2, the vertex $x_i$ has $u_i$ as its only neighbor in $I$, so the one-vertex path consisting of $x_i$ satisfies the following property: every vertex of the path has $u_i$ as its only neighbor in $I$. We may therefore extend this path backwards along $\mathfrak{P}$ as long as this property remains true. Let $T_i$ be the resulting maximal subpath of $\mathfrak{P}$ ending in $x_i$, and let $y_i$ be the first vertex of $T_i$. If $y_i=v_1$, then $y_i$ has no predecessor on the path $\mathfrak{P}$. Otherwise, let $q_i$ be the predecessor of $y_i$. Let $B' \subseteq B$ be a maximal subset of $B$ such that $q_i$ exists for each $i \in B'$. We have $|B'| \geq |B|-1$, since $B'$ contains all indices of $B$ except potentially $i=1$. For each $i \in B'$, the vertex $q_i$ has a neighbor in $I \setminus \{u_i\}$ by the definition of $T_i$. Let $u_{t(i)} \in I \setminus \{u_i\}$ be such a neighbor of $q_i$. Then $\{u_{t(i)},q_i,y_i\}$ induces a $P_3$, since $y_i$ has only one neighbor in $I$, namely $u_i$. Let $Y := \{y_i ~|~ i \in B'\}$. We now claim that $Y$ is an independent set.

\bigskip 

\noindent \textbf{Claim 5.} $Y$ is an independent set.

\bigskip

\noindent

Suppose for the sake of obtaining a contradiction that $y_iy_j \in E(G)$ for two distinct integers $i, j \in B'$. Then $\{u_i,y_i,y_j\}$ induces a $P_3$, since $y_j$ has only one neighbor in $I$, namely $u_j$. But then $G[\{u_i,y_i,y_j\} \cup (I \setminus \{u_i,u_j\})]$ contains an induced $P_3+sP_1$, since $k-2 \geq s$. Contradiction. This proves the claim. 

\bigskip

We now show that $q_i$ is adjacent to all but at most $s$ vertices of $Y \setminus \{y_i\}$.

\bigskip

\noindent \textbf{Claim 6.} Let $i \in B'$. Then $q_i$ is adjacent to at least $|Y|-1-s$ vertices in $Y \setminus \{y_i\}$.

\bigskip

\noindent

Recall that $\{u_{t(i)},q_i,y_i\}$ induces a $P_3$. Let $Y'=Y \setminus \{y_{t(i)},y_i\}$ if $t(i) \in B'$ and otherwise let $Y'=Y \setminus \{y_i\}$. Let $j \in B' \setminus \{i,t(i)\}$. Suppose $q_iy_j \notin E(G)$. Then $y_j$ is not adjacent to $q_i$ by assumption, it is not adjacent to $u_{t(i)}$ since $y_j$ has only one neighbor in $I$ (namely, $u_j$) and $y_j$ is not adjacent to $y_i$ because of Claim 5. This proves the claim, because if $q_i$ would be adjacent to strictly fewer than $|Y|-1-s$ vertices in $Y \setminus \{y_i\}$, then $G[\{u_{t(i)},q_i,y_i\} \cup Y']$ would contain an induced $P_3+sP_1$.

\bigskip

Finally, we claim that $|B|$ is bounded by a function of $s$.

\bigskip

\noindent \textbf{Claim 7.} We have $|B| \leq 4s+3$.

\bigskip

\noindent

Suppose for the sake of obtaining a contradiction that $|B| \geq 4s+4$. Then $|B'|=|Y| \geq 4s+3$ by $|B'| \ge |B| - 1$. We will now build an auxiliary graph $A$ on the vertex set $Y$. There is an edge between $y_i$ and $y_j$ in $A$ if and only if $q_iy_j \in E(G)$ and $q_jy_i \in E(G)$. By Claim 6, each $q_i$ ($i \in B'$) is adjacent to at least $|Y|-s$ vertices in $Y$. So the number of ordered pairs $(i,j)$ such that $i\neq j$ and $i,j \in B'$ and $q_iy_j \notin E(G)$ is at most $s|Y|$. Now $y_iy_j \notin E(A)$ if and only if $q_iy_j \notin E(G)$ or $q_jy_i \notin E(G)$. Therefore, $|E(A)| \geq \binom{|Y|}{2}-s|Y|$. Since $|Y| \geq 4s+3$, we have $\binom{|Y|}{2}-s|Y|>\frac{|Y|^2}{4}$, because $\binom{|Y|}{2}-s|Y|-\frac{|Y|^2}{4}=\frac{|Y|(|Y|-4s-2)}{4}>0$. By Mantel's Theorem, $A$ contains a triangle between three vertices $y_{\alpha}, y_{\beta}, y_{\gamma}$, where $\alpha$, $\beta$ and $\gamma$ are distinct elements in $B'$. This means that $E(G)$ contains the six edges $q_{\alpha}y_{\beta}, q_{\alpha}y_{\gamma}, q_{\beta}y_{\alpha}, q_{\beta}y_{\gamma},
q_{\gamma}y_{\alpha}, q_{\gamma}y_{\beta}$. Now relabel these indices as $i,j,k$ according to their order on $\mathfrak{P}$ (i.e., such that $q_i, y_i \ldots, q_j, y_j, \ldots, q_k, y_k$ appear in that order on $\mathfrak{P}$). This means that $\mathfrak{P}$ has the form $L, q_i, y_i,C, q_j, y_j, D, q_k, y_k, R$, where $L, C, D$ and $R$ are (potentially empty) subpaths of $\mathfrak{P}$. Replacing the subpath $q_i, y_i, C, q_j, y_j, D, q_k, y_k$
of $\mathfrak{P}$ by $q_i, y_j, D, q_k, y_i, C, q_j, y_k$
and leaving the rest of $\mathfrak{P}$ unchanged gives a second Hamiltonian $v_1v_n$-path $\mathfrak{P}'$ in $G$ (see Fig.~\ref{fig:claim7-rerouting}). This leads to a contradiction and proves the claim.

\begin{figure}[ht]
\centering
\begin{tikzpicture}[
    x=1.3cm,y=1cm,
    block/.style={draw,rounded corners,minimum width=1.0cm,minimum height=.55cm},
    dot/.style={circle,fill,inner sep=1.4pt},
    >=Latex,scale=0.8
]
    \node[block] (L)  at (0,0) {$L$};
    \node[dot,label=below:{$q_i$}] (qi) at (1.5,0) {};
    \node[dot,label=below:{$y_i$}] (yi) at (2.5,0) {};
    \node[block] (C)  at (4,0) {$C$};
    \node[dot,label=below:{$q_j$}] (qj) at (5.5,0) {};
    \node[dot,label=below:{$y_j$}] (yj) at (6.5,0) {};
    \node[block] (D)  at (8,0) {$D$};
    \node[dot,label=below:{$q_k$}] (qk) at (9.5,0) {};
    \node[dot,label=below:{$y_k$}] (yk) at (10.5,0) {};
    \node[block] (R)  at (12,0) {$R$};

    \draw[ultra thick] (L) -- (qi) -- (yi) -- (C) -- (qj) -- (yj) -- (D) -- (qk) -- (yk) -- (R);

    \draw[blue] (qi) to[bend left=22] (yj);
    \draw[blue] (qk) to[bend right=22] (yi);
    \draw[blue] (qj) to[bend left=22] (yk);

    \begin{scope}[yshift=-100.0]
    
    \node[block] (L)  at (0,0) {$L$};
    \node[dot,label=below:{$q_i$}] (qi) at (1.5,0) {};
    \node[dot,label=below:{$y_i$}] (yi) at (2.5,0) {};
    \node[block] (C)  at (4,0) {$C$};
    \node[dot,label=below:{$q_j$}] (qj) at (5.5,0) {};
    \node[dot,label=below:{$y_j$}] (yj) at (6.5,0) {};
    \node[block] (D)  at (8,0) {$D$};
    \node[dot,label=below:{$q_k$}] (qk) at (9.5,0) {};
    \node[dot,label=below:{$y_k$}] (yk) at (10.5,0) {};
    \node[block] (R)  at (12,0) {$R$};

    \draw[ultra thick,red] (L) -- (qi);
    \draw[ultra thick,red] (yi) -- (C) -- (qj);
    \draw[ultra thick,red] (yj) -- (D) -- (qk);
    \draw[ultra thick,red] (yk) -- (R);
    \draw[] (qi) -- (yi);
    \draw[]  (qj) -- (yj);
    \draw[]  (qk) -- (yk);
    
    \draw[ultra thick,red] (qi) to[bend left=22] (yj);
    \draw[ultra thick,red] (qk) to[bend right=22] (yi);
    \draw[ultra thick,red] (qj) to[bend left=22] (yk);
    
    \end{scope}

    \node[draw=none,fill=none] at (6,-1.5) {$\rightarrow$};
\end{tikzpicture}
\caption{The rerouting used in Claim~7. The blue edges are guaranteed by the triangle in the auxiliary graph $A$ and they yield a second Hamiltonian $v_1v_n$-path.}
\label{fig:claim7-rerouting}
\end{figure}
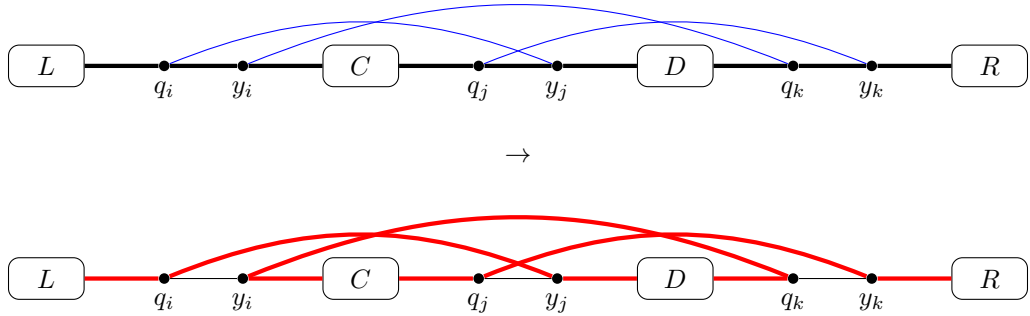

\bigskip

By combining Claim 4 and Claim 7, we get $k=\alpha(G) \leq 5s+6$, as desired.
\end{proof}

If $G$ is a $P_3+sP_1$-free graph with a unique Hamiltonian path between $a, b \in V(G)$, then $G$ is $(5s+7)P_{1}$-free by Theorem~\ref{th:mainTheorem}. By combining this fact with Theorem~\ref{th:sP1} and Lemma~\ref{lem:pathToCycles}, we get the following corollaries.

\begin{corollary}
\label{cor:P3PlussP12Switch}
    Let $s \geq 0$ be an integer and let $G$ be a $P_3+sP_1$-free graph on $R(5s+7,5s+10)$ vertices having a Hamiltonian path $\mathfrak{P}$ between two distinct vertices $a, b \in V(G)$. Then $G$ has a second Hamiltonian $ab$-path.
\end{corollary}

\begin{corollary}
\label{cor:countingsP3PlussP1}
Let $s \geq 0$ be an integer and let $G$ be a $P_3+sP_1$-free graph on $n$ vertices having two vertices $a,b \in V(G)$ between which there exists a Hamiltonian path. Then $G$ contains at least $2^{\left\lfloor \frac{n}{R(5s+7,5s+10)}\right\rfloor}$ distinct Hamiltonian $ab$-paths.
\end{corollary}

The following observation follows from the classical work of Erd{\H{o}}s~\cite{E47} and Erd{\H{o}}s-Szekeres~\cite{ES35} and shows that we were indeed able to improve the exponent from Section~\ref{sec:expCrudeBound}.
\begin{observation}
\label{obs:Ramsey2ThetaS}
    We have $R(5s+7,5s+10)=2^{\Theta(s)}$.
\end{observation}

By summarizing Observation~\ref{obs:infFam1}, Observation~\ref{obs:infFam2}, Observation~\ref{obs:infFam3}, Corollary~\ref{cor:countingsP3PlussP1} and Observation~\ref{obs:Ramsey2ThetaS}, we obtain Theorem~\ref{th:summary} from the Introduction.

\SummaryTheorem*

\section{Number of Hamiltonian cycles in Hamiltonian $F$-free graphs and algorithmic consequences}
\label{sec:hamCycles}

Another natural question to ask is the following one: Given a Hamiltonian $F$-free graph on $n$ vertices, how many Hamiltonian cycles must this graph have (in terms of $n$)? Since a Hamiltonian graph has a Hamiltonian path between two adjacent vertices, Observation~\ref{obs:infFam1}, Observation~\ref{obs:infFam2}, Corollary~\ref{cor:countingsP3PlussP1}, and Observation~\ref{obs:Ramsey2ThetaS} yield the following corollary.

\begin{corollary}
    Let $F$ be a graph. If $F$ is not an induced subgraph of $P_4+sP_1$ 
    for any integer $s \geq 0$, then there exists an infinite family of uniquely Hamiltonian $F$-free graphs. On the other hand, if $F$ is an induced subgraph of $P_3+sP_1$ for some integer $s \geq 0$, then there exists a constant $c_s=2^{\Theta(s)}$ such that every Hamiltonian $F$-free graph contains at least $2^{\left\lfloor \frac{n}{c_s} \right\rfloor}$ distinct Hamiltonian cycles.
\end{corollary}

Hence, the only case which is open is the case where $F=P_4+sP_1$ for some integer $s \geq 0$. We remark that for $s=0$, there are no uniquely Hamiltonian $P_4$-free graphs on $n \geq 5$ vertices.

\begin{observation}
    Let $G$ be a Hamiltonian $P_4$-free graph on $n \geq 5$ vertices. Then $G$ is not uniquely Hamiltonian.
\end{observation}
\begin{proof}
    Let $\mathfrak{C}=v_1 v_2 \ldots v_n v_1$ be a Hamiltonian cycle in $G$. Since $G$ is $P_4$-free, it is well-known that $G$ contains two distinct vertices $v_i, v_j \in V(G)$ (with $i<j$) such that $N(v_i)=N(v_j)$ or $N[v_i]=N[v_j]$ (i.e., either the open or closed neighborhoods are identical)~\cite{CLB81}. But then the cycle obtained by swapping $v_i$ and $v_j$, namely $\mathfrak{C}'=v_1 v_2 \ldots v_{i-1} v_j v_{i+1} \ldots v_{j-1} v_i v_{j+1} \ldots v_n v_1$ yields a second Hamiltonian cycle.
\end{proof}




We also wish to state some algorithmic consequences of our results. The problem \textsc{Second Hamiltonian cycle} takes as input a graph $G$ and a Hamiltonian cycle in $G$ and asks to produce a second Hamiltonian cycle in $G$ or state that it does not exist. This problem was popularized by Papadimitriou~\cite{P94} and mainly received attention for classes of graphs where a second Hamiltonian cycle is guaranteed to exist. For example, Hamiltonian graphs in which every vertex has an odd degree are not uniquely Hamiltonian due to a theorem of Smith (see~\cite{T46}). The \textsc{Second Hamiltonian cycle} problem received attention for example for the class of cubic Hamiltonian graphs~\cite{DMSZ20} and the class of bipartite Pfaffian graphs of minimum degree at least three~\cite{BKN24}. We consider the restriction to $F$-free graphs: the problem \textsc{Second Hamiltonian cycle in $F$-free graphs} takes as input an $F$-free graph $G$ and a Hamiltonian cycle in $G$ and asks to produce a second Hamiltonian cycle in $G$ or state that it does not exist. Theorem~\ref{th:sP1} has the following algorithmic consequence.

\begin{proposition}
    Let $s \geq 1$ be an integer. The problem \textsc{Second Hamiltonian cycle in $sP_1$-free graphs} can be solved in $O(|V(G)|+|E(G)|)$ time, where $G$ is the input graph.
\end{proposition}
\begin{proof}
  Let $G$ be the graph from the input (on $n$ vertices) and let $\mathfrak{C}:=v_1 v_2 \ldots v_n v_1$ be the Hamiltonian cycle from the input. By Erd{\H{o}}s-Szekeres~\cite{ES35} we have $R(s,s+3)<4^{s+2}$. We now describe an algorithm for solving \textsc{Second Hamiltonian cycle in $sP_1$-free graphs} in $O(|V(G)|+|E(G)|)$ time. If $n<4^{s+2}$, the algorithm can exhaustively search over all Hamiltonian cycles in $G$ using a brute-force approach. Since $s$ is a constant, this can be done in constant time. If $n \geq 4^{s+2}$, consider $R(s,s+3)$ consecutive vertices on the given Hamiltonian cycle, viewed as a Hamiltonian path in the induced subgraph on those vertices. This induced subgraph is still $sP_1$-free. By Theorem~\ref{th:sP1}, this induced subgraph has an alternative Hamiltonian path between its endpoints obtained by a 2-switch. This also gives rise to a second Hamiltonian cycle $\mathfrak{C}'$ in $G$, obtained from $\mathfrak{C}$ by a 2-switch. Therefore, the algorithm can iterate over each consecutive pair of vertices $v_i, v_{i+1}$, record all edges incident with $v_i$ and $v_{i+1}$ and verify whether there exists a pair of edges $(v_iv_j,v_{i+1}v_{j+1})$ to construct~$\mathfrak{C}'$.
\end{proof}

\section{Opportunities for future work}
We conclude this paper by listing directions for future work that spark our interest. The first problem is related to the asymptotic behavior of the minimal number of distinct Hamiltonian paths between two vertices with at least one such path in $F$-free graphs.

\begin{problem}
Let $s \geq 0$ be an integer and let $F$ be an induced subgraph of $P_3+sP_1$.
For an $n$-vertex graph $G$ and vertices $a,b\in V(G)$, let $h_{a,b}(G)$ denote the
number of Hamiltonian paths in $G$ with endpoints $a$ and $b$. Define the quantity
$$
m_F(n)=
\min \bigl\{
h_{a,b}(G):
|V(G)|=n,\; G \text{ is } F\text{-free},\;a,b \in V(G),\;
h_{a,b}(G)\ge 1
\bigr\}.
$$
Determine a function $g(F,n)$ such that
$$
m_F(n)=\Theta(g(F,n)).
$$
\end{problem}

In the current paper, we consider $F$-free graphs. The next natural step is to consider similar questions for graphs that are $F_i$-free for each graph $F_i$ in a family of forbidden graphs $\mathcal{F}$.
\begin{question}
    For which families of graphs $\mathcal{F}$ does there exist an infinite family of graphs $\mathcal{G}$ such that each graph $G \in \mathcal{G}$ has two vertices between which there exists a unique Hamiltonian path and such that $G$ is $F_i$-free for every $F_i \in \mathcal{F}$?
\end{question}

In the current paper, we gave a complete dichotomy for the case where $|\mathcal{F}|=1$ and the next natural target would be to consider the case where $|\mathcal{F}|=2$.

Finally, motivated by our observations in Section~\ref{sec:hamCycles} we formulate the following conjecture.
\begin{conjecture}
    Let $s \geq 1$ be an integer and let $G$ be a Hamiltonian $P_4+sP_1$-free graph on $n$ vertices. There exists a constant $n_s$ depending on $s$ such that $G$ is not uniquely Hamiltonian if $n \geq n_s$.
\end{conjecture}

\section*{Acknowledgements}
\noindent Jorik Jooken is supported by a Postdoctoral Fellowship of the Research Foundation Flanders (FWO) with grant number 1222524N. 


\end{document}